\documentclass[12pt, a4paper]{amsart}
\usepackage{amssymb}
\usepackage{tikz} 
\usepackage{mathabx}
\usepackage{graphicx}
\usepackage{mathrsfs}
\usepackage{hyperref}
\usepackage{import}
\usepackage{tcolorbox}

\theoremstyle{plain}

\newtheorem{theorem}{Theorem}[section]
\newtheorem{lemma}[theorem]{Lemma}
\newtheorem{proposition}[theorem]{Proposition}
\newtheorem{corollary}[theorem]{Corollary}

\theoremstyle{definition}
\newtheorem{define}{Definition}[section]

\theoremstyle{remark}
\newtheorem{remark}{Remark}[section]

\newtheorem*{acknowledgement}{Acknowledgement}

\begin{document}

\date\today

\title[Exhaustion of hyperbolic complex manifolds]{Exhaustion of hyperbolic complex manifolds and relations to the squeezing function}
\author{Ninh Van Thu, Trinh Huy Vu and Nguyen Quang Dieu\textit{$^{1,2}$}}

\address{Ninh Van Thu}
\address{~School of Applied Mathematics and Informatics, Hanoi University of Science and Technology, No. 1 Dai Co Viet, Hai Ba Trung, Hanoi, Vietnam}
\email{thu.ninhvan@hust.edu.vn}

\address{Trinh Huy Vu}
\address{~Department of Mathematics, Vietnam
National University at Hanoi, 334 Nguyen Trai, Thanh Xuan, Hanoi, Vietnam}
\email{trinhhuyvu1508@gmail.com}

\address{Nguyen Quang Dieu}
\address{\textit{$^{1}$}~Department of Mathematics, Hanoi National University of Education, 136 Xuan Thuy, Cau Giay, Hanoi, Vietnam}
 \address{\textit{$^{2}$}~Thang Long Institute of Mathematics and Applied Sciences,
Nghiem Xuan Yem, Hoang Mai, Hanoi, Vietnam}
\email{ngquang.dieu@hnue.edu.vn}

\subjclass[2020]{Primary 32H02; Secondary 32M05, 32F18.}
\keywords{Hyperbolic complex manifold, exhausting sequence, $h$-extendible domain}

\begin{abstract}
The purpose of this article is twofold. The first aim is to characterize an $n$-dimensional Kobayashi hyperbolic complex manifold $M$ exhausted by a sequence $\{\Omega_j\}$ of domains in $\mathbb C^n$ via an exhausting sequence $\{f_j\colon \Omega_j\to M\}$ such that $f_j^{-1}(a)$ converges to a boundary point $\xi_0 \in \partial \Omega$ for some point $a\in M$. Then, our second aim is to show that any spherically extreme boundary point must be strongly pseudoconvex. 
\end{abstract}

\maketitle

\section{introduction}
Let $M$ be an $n$-dimensional Kobayashi hyperbolic complex manifold. Let $\{M_j\}$ and $\{\Omega_j\}$ be two sequences of open subsets in $M$ and $\mathbb C^n$ respectively. Suppose that $M$ \emph{can be exhausted by} $\{\Omega_j\}$ via an exhausting sequence $\{f_j\colon \Omega_j \to M_j\subset M\}$ in the sense that $f_j$ is a biholomorphism from $\Omega_j$ onto $M_j$ for all $j\geq 1$ and $\displaystyle\cup_{j=1}^\infty M_j=M$. Then it is a natural problem is to describe $M$ in terms of $\{\Omega_j\}$.  In the case that $\Omega_j=\Omega$ for all $j\geq 1$, this problem is called the \emph{union problem} (cf. \cite{BBMV21, FSi81}). In $1977$, J. E. Forn{\ae}ss and L. Stout \cite{FS77} proved that if $\Omega$ is the unit ball $\mathbb B^n:=\{z\in \mathbb C^n\colon |z|<1\}$, then $M$ is biholomorphically equivalent to $\mathbb B^n$. More generally, $M$ is biholomorphically equivalent to $\Omega$ if it is a homogeneous bounded domain in $\mathbb C^n$. For further results about the union problem the reader may also consult the references \cite{FSi81, FM95, Liu18, NT21, BBMV21}.

Now let us fix a point $a\in M$ and assume that $\lim \Omega_j=\Omega$ (see Definition \ref{Def1} for the notion of this limit). Then, we consider the behavior of  the sequence $\{f_j^{-1}(a)\}\subset \Omega$. In the case when $\{f_j^{-1}(a)\}$ converges to a point $p\in \Omega$, $M$ is biholomorphically equivalent to $\Omega$ (cf.  Corollary \ref{key-prop} in Section \ref{S2}). Therefore, we especially pay attention to the case that $\{f_j^{-1}(a)\}$ converges to a boundary point $\xi_0\in \partial\Omega$.

In the first part of this paper, we give a characterization of our manifold $M$ in term of the behavior of the orbit $\{f_j^{-1}(a)\}$. To do this, let us fix positive integers $m_1,\ldots, m_{n-1}$ and let $P(z')$ be a $(1/m_1,\ldots,1/m_{n-1})$-homogeneous polynomial given by 
\begin{equation*}\label{series expression}
P(z')=\sum_{wt(K)=wt(L)=1/2} a_{KL} {z'}^K  \bar{z}'^L,
\end{equation*}
where $a_{KL}\in \mathbb C$ with $a_{KL}=\bar{a}_{LK}$, satisfying that $P(z')>0$ whenever $z'\ne 0$. Here and in what follows, $z':=(z_1,\ldots,z_{n-1})$ and $ wt(K):=\sum_{j=1}^{n-1} \frac{k_j}{2m_j}$ denotes the weight of any multi-index $K=(k_1,\ldots,k_{n-1})\in \mathbb N^{n-1}$ with respect to $\Lambda:=(1/m_1,\ldots,1/m_{n-1})$. Then the general ellipsoid $D_P$ in $\mathbb C^{n}\;(n\geq 2)$, defined in  \cite{NNTK19} by
\begin{equation*}
\begin{split}
D_P &:=\{(z',z_n)\in \mathbb C^{n}\colon |z_n|^2+P(z')<1\}.
\end{split}
\end{equation*}
We emphasize here that 
\begin{equation}\label{eq01}
P(a^{1/m_1} z_1, a^{1/m_2}z_2,\ldots, a^{1/m_{n-1}}z_{n-1}) =P(z'),\; \forall z'\in \mathbb C^{n-1}, \forall a\in \mathbb C\setminus\{0\}.
\end{equation}
Hence, the automorphism group of $D_P$ (denoted by $\mathrm{Aut}(D_P)$) contains the automorphisms $\phi_{a,\theta}$, defined by
\begin{equation*}
(z',z_n)\mapsto  \left( \frac{(1-|a|^2)^{1/2m_1}}{(1-\bar{a}z_n)^{1/m_1}} z_1,\ldots,  \frac{(1-|a|^2)^{1/2m_{n-1}}}{(1-\bar{a}z_n)^{1/m_{n-1}}} z_{n-1}, e^{i\theta} \frac{z_n-a}{1-\bar{a}z_n}\right),
\end{equation*}
where $a\in\Delta:=\{z\in \mathbb C\colon |z|<1\}$ and $\theta\in \mathbb R$ (see Lemma \ref{dilation-property} in Section \ref{S2}).

Throughout this paper, we assume that the domain $D_P$ is a WB-domain, i.e.,  $D_P$ is strongly pseudoconvex at every boundary point outside the set $\{(0',e^{i\theta})\colon \theta\in \mathbb R\}$ (cf. \cite{AGK16}).

For any $s, r\in (0,1]$,  inspired by \cite[Lemma $2.5$]{Liu18}  let us define $D^s_P, D^{s,r}_P$,  and $D_{P,r}$ respectively by
\begin{equation*}
	\begin{split}
D^s_P&:=\{z\in \mathbb C^n \colon |z_n-b|^2+sP(z')<s^2\};\\
D^{s,r}_P&:=\{z\in \mathbb C^n \colon |z_n-b|^2+\dfrac{s}{r}P(z')<s^2\};\\
D_{P,r}:=D_{P/r}&=\Big\{z\in \mathbb C^{n}\colon |z_n|^2+ \dfrac{1}{r} P(z')<1\Big\},
\end{split}
\end{equation*}
where $s=1-b$. We note that $D_P^{s,r}\subset D_P^s\subset D_P$ (see Remark \ref{re-3.1}), $D_P^{s,1}=D_P^s$, and the property that $\lim \psi_j^{-1}(D^s_P)=D_P$ for a certain family $\{\psi_j\}\subset \mathrm{Aut}(D_P)$ (cf. Lemma \ref{exhaustion-ellipsoid} in Section \ref{S4}) plays a key role in the proofs of our main theorems below. 

In what follows, let us denote by $\Delta$ the unit disc in $\mathbb C$ and for a sequence $\{a_j\}\subset \Delta$ converging to $1\in\partial \Delta$ we always denote by $x_j:=1-\mathrm{Re}(a_j)$ and $y_j:=\mathrm{Im}(a_j)$ for $j\geq 1$. To state our first result, we need the following definition (see Section \ref{S4} for its application).
\begin{define}\label{def1} We say that the sequence $\{a_j\}\subset \Delta$ converges \emph{parabolically nontangentially} to $1$ in $\Delta$ if $\displaystyle \lim_{j\to\infty} \frac{y_j^2}{x_j}=0$.
\end{define}

Our first main result is the following theorem.
\begin{theorem}\label{maintheorem2} Let $M$ be an $n$-dimensional Kobayashi hyperbolic complex manifold and let $\Omega$ be a pseudoconvex domain in $\mathbb{C}^{n}$.  Let $\{\Omega_j\}$ be a sequence of subdomains of $D_P$ such that $D^s_P\subset \Omega_j\subset D_P, j\geq 1$. Suppose also that $M$ can be exhausted by $\{\Omega_j\}$ via an exhausting sequence $\{f_j\colon D_P\supset \Omega_j \to M_j\subset M\}$. If there exist a point $a \in M$ and a number $r\in (0,1)$ such that the sequence $D^{s,r}_P\ni \eta_j =(\eta_j,\eta_{jn}) := f_j^{-1}(a)$ converges to $(0',1)$ in $D_P$ and $\{\eta_{jn}\}\subset \Delta$ converges parabolically nontangentially to $1\in \partial \Delta$, then $M$ is biholomorphically equivalent to $D_P$.
\end{theorem}
\begin{remark} Notice that the point $p=(0',1)$ is $(P,s)$-extreme for each domain $\Omega_j$ (cf. \cite[Definition $1.1$]{NNC21} for the notion of $(P,s)$-extreme points) and the convergence of a sequence of points in $D_P^{s,r}$ to $p$, with its $n$-th component converging nontangentially to $1\in\partial \Delta$, is exactly the $\Lambda$-nontangential convergence introduced in \cite[Definition $3.4$]{NN19}. Therefore, in this situation \cite[Theorem $1.1$]{NT21} yields Theorem \ref{maintheorem2} for the case that $\Omega_j=D_P$ for all $j\geq 1$. However, our proof here is quite different  and more simple. In particular, we do not need to assume the convergence of $\Omega_j$ to $D_P$ and the proof shows that $D_P$ may not be a WB-domain.
\end{remark}
\begin{remark} The parabolically nontangential convergence requirement appears necessary, because without this condition our manifold $M$ may not be biholomorphically equivalent to $D_P$ (cf. Remark \ref{remark2023-1}).

\end{remark}

Now we move to the definition of $\Lambda$-tangential convergence.
\begin{define} \label{def2} We say that $\{\eta_j\}\subset D_P\cap U$ converges \emph{$\Lambda$-tangentially} to $p=(0',1)$ if for any $0<r<1$ there exists $j_{r}\in \mathbb N$ such that $\eta_{j}\not \in D_{P,r}$ for all $j\geq j_{r}$.
\end{define}

For the case that $\{\eta_j\}\subset \Omega\cap U$ converges $\Lambda$-tangentially to $p=(0',1)$, we do not know whether $M$ is biholomorphically equivalent to $D_P$. However, the following theorem shows that $M$ is biholomorphically equivalent to the unit ball $\mathbb{B}^{n}$ provided  that all $\partial\Omega_j$ share a small neighborhood of the point $(0',1)$ with $\partial D_P$ to which the sequence of points converges $\Lambda$-tangentially. More precisely, our second main result is the following theorem.
\begin{theorem}\label{maintheorem10}Let $M$ be an $n$-dimensional Kobayashi hyperbolic complex manifold.  Let $\{\Omega_j\}$ be a sequence of subdomains of $D_P$ such that $\Omega_j\cap U=D_P\cap U$, $j\geq 1$, for a fixed neighborhood $U$ of $(0',1)$ in $\mathbb C^n$. Suppose also that $M$ can be exhausted by $\{\Omega_j \}$ via an exhausting sequence $\{f_j\colon D_P\supset \Omega_j \to M_j\subset M\}$. If there exists a point $a \in M$ such that the sequence $\eta_j:= f_j^{-1}(a)$ converges $\Lambda$-tangentially to $(0',1)$ in $D_P$, then $M$ is biholomorphically equivalent to the unit ball $\mathbb{B}^{n}$.
\end{theorem}
\begin{remark}
Theorem \ref{maintheorem2} and Theorem \ref{maintheorem10} may be summarized roughly in the following dichotomy: if the convergence of $f_j^{-1} (a)$ to $p$ is $\Lambda$-nontangential (see page 2) then $M$ is biholomorphically equivalent to $D_P$, otherwise if the convergence is $\Lambda$-tangential then $M$ is biholomorphically equivalent to the ball $\mathbb B^n$.
\end{remark}
\begin{remark}
In order to prove Theorem \ref{maintheorem10}, we shall show without loss of generality that the sequence $\{\psi_j^{-1}(\eta_j)\}$ converges nontangentially to some boundary point $p\in \partial D_P\cap \{z_n=0\}$. Thanks to the fact that $D_P$ is a WB-domain, that point is strongly pseudoconvex and hence the proof easily follows from Theorem \ref{togetmodelstronglypsc1}. However, if $D_P$ is not a WB-domain, i.e. the point $p$ may not be strongly pseudoconvex, but it is $h$-extendible in the sense of J. Yu (cf. \cite{Yu95}), then $M$ would be biholomorphically equivalent to some model by Theorem $1.1$ in \cite{NT21}.      
\end{remark}

Let $\Omega$ be a pseudoconvex domain in $\mathbb C^n$. Suppose that $\xi_0\in \partial \Omega$ is a boundary orbit accumulation point, i.e. there exists a sequence $\{\varphi_j\}\subset \mathrm{Aut}(\Omega)$ such that $\eta_j:=\varphi_j(a)$ converges to $\xi_0\in \partial\Omega$ for some point $a\in \Omega$. Here and in what follows, $\mathrm{Aut}(\Omega)$ denotes the automorphism group of $\Omega$. Then, one notices that our domain $\Omega$ is exhausted by itself via the sequence $\varphi_j\colon \Omega\to\Omega$. In addition, as an application of Theorem \ref{maintheorem2} and Theorem \ref{maintheorem10}, we show that if $\Omega$ is a subdomain of $D_P$ with $\Omega\cap U=D_P\cap U$ for a fixed neighborhood $U$ of $(0',1)$ in $\mathbb C^n$, then $\Omega$ must be biholomorphically equivalent to either $D_P$ or $\mathbb B^n$ (cf. Corollary \ref{corollary-model} in Section \ref{S4}).

Now we move to the second part of this paper. Let $\Omega$ be a bounded domain in $\mathbb C^n$ and $q \in \Omega$. For a holomorphic embedding $f\colon \Omega \to \mathbb B^n:=\mathbb B(0;1)$ with $f(q)=0$, one sets
$$
\sigma_{\Omega,f}(q):=\sup\left \{r>0\colon B(0;r)\subset f(\Omega)\right\},
$$
where $\mathbb B^n (z;r)\subset\mathbb{C}^n$  denotes the ball of radius $r$ with center at $z$. Then the \textit{squeezing function} $\sigma_{\Omega}: \Omega\to\mathbb R$ is defined in \cite{DGF12} as
$$
\sigma_{\Omega}(q):=\sup_{f} \left\{\sigma_{\Omega,f}(q)\right\}.
$$
Note that $0 < \sigma_{\Omega}(z)\leq 1$ for any $z \in \Omega$ and it is obvious that the squeezing function is invariant under biholomorphisms.

Now we recall the definition of spherically extreme boundary points (cf. \cite{KZ16}). Indeed, a boundary point $p\in \partial \Omega$ is said to be locally spherically extreme if there exist a neighborhood $U$ of $p$ and a ball $\mathbb B(c(p);R)$ in $\mathbb C^n$ of some radius $R$, center at some point $c(p)$ such that $\partial \Omega\cap U$ is $\mathcal{C}^2$-smooth, $\Omega\cap U\subset \mathbb B(c(p);R)$, and $p\in \partial \Omega\cap \partial\mathbb B(c(p);R)$.

\begin{theorem}\label{togetmodelstronglypsc} Let $M$ be an $n$-dimensional Kobayashi hyperbolic complex manifold and let $\Omega$ be a pseudoconvex domain in $\mathbb{C}^{n}$. Suppose that $\partial\Omega$ admits a spherically extreme boundary point $\xi_0$ in a neighborhood of which the boundary $\partial\Omega$ is $\mathcal{C}^2$-smooth. In addition, let $\{\Omega_j\}$ be a subdomains of $\Omega$  such that $\Omega_j\cap U=\Omega\cap U$, $j\geq 1$, for some neighborhood $U$ of $\xi_0$ in $\mathbb C^{n}$. Suppose also that $M$ can be exhausted by $\{\Omega_j \}$ via an exhausting sequence $\{f_j: \Omega\supset \Omega_j \to M_j\subset M\}$. If there exists a point $a \in M$ such that the sequence $\eta_j := f_j^{-1}(a)$ converges to $\xi_0$, then $M$ is biholomorphically equivalent to the unit ball $\mathbb{B}^{n}$.
\end{theorem}

By using the scaling method, K.-T. Kim and L. Zhang \cite[Theorem $3.1$]{KZ16} proved that if a domain in $\mathbb C^n$ admits a locally spherically extreme boundary point $p$, then 
$$      
\lim\limits_{\Omega\cap U\ni q\to p}\sigma_{\Omega\cap U}(q)=1, 
$$ 
where $U$ is a small neighborhood of $p$. Of course, we may not have that $\lim\limits_{\Omega\ni q\to p}\sigma_{\Omega}(q)=1$ (see \cite[Theorem $1$]{FN21}). It is known that every strongly convex boundary point is locally spherically extreme. Conversely, we will show in the last section 
(Proposition \ref{spherical-strongly-psc}) that every locally spherically extreme point is actually strongly pseudoconvex.

As an application of Proposition \ref{spherical-strongly-psc}, in the union problem, the Kobayashi hyperbolic complex manifold $M$ must be biholomorphically equivalent to $\mathbb B^n$ provided that $\{f_j^{-1}(a)\}$ converges to a spherically extreme boundary point $\xi_0\in \partial\Omega$ (see Theorem \ref{togetmodelstronglypsc1} and the proof of Theorem \ref{togetmodelstronglypsc} in Sections \ref{S4} and \ref{S5} respectively for more details). More generally, if $\lim\limits_{z\to \xi_0}\sigma_\Omega(z)=1$, then $M$ is also biholomorphically equivalent to $\mathbb B^n$ (cf. Corollary \ref{hhr1}).   

The organization of this paper is as follows: In Section~\ref{S2} we provide some results concerning the normality of  a sequence of biholomorphisms. Next, we give our proofs of Theorem \ref{maintheorem2} and Theorem \ref{maintheorem10} in Section \ref{S4}. Finally, the proof of Theorem \ref{togetmodelstronglypsc} will be introduced in Section \ref{S5}.

\section{The normality}\label{S2}
First of all, we recall the following definition (see \cite{GK} or \cite{DN09}). 
\begin{define}\label{Def1} Let $\{\Omega_i\}_{i=1}^\infty$ be a sequence of open sets in a complex manifold $M$ and $\Omega_0 $ be an open set of $M$. The sequence $\{\Omega_i\}_{i=1}^\infty$ is said to converge to $\Omega_0 $ (written $\lim\Omega_i=\Omega_0$) if and only if 
        \begin{enumerate}
                \item[(i)] For any compact set $K\subset \Omega_0,$ there is an $i_0=i_0(K)$ such that $i\geq i_0$ implies that $K\subset \Omega_i$; and 
                \item[(ii)] If $K$ is a compact set which is contained in $\Omega_i$ for all sufficiently large $i,$ then  $K\subset \Omega_0$.
        \end{enumerate}  
\end{define}

Next, we recall the following proposition, which is a generalization of the theorem of H. Cartan (see \cite{DN09, GK, TM}).
\begin{proposition} \label{T:7}  Let $\{A_i\}_{i=1}^\infty$ and $\{\Omega_i\}_{i=1}^\infty$ be sequences of domains in complex manifolds $M$ and $N$ respectively with $\dim M=\dim N$, $\lim A_i=A_0$, and $\lim \Omega_i=\Omega_0$ for some (uniquely determined) domains $A_0$ in $M$ and $\Omega_0$ in $N$. Suppose that $\{f_i: A_i \to \Omega_i\} $ is a sequence of biholomorphic maps. Suppose also that the sequence $\{f_i: A_i\to N \}$ converges uniformly on compact subsets of $ A_0$ to a holomorphic map $F:A_0\to N $ and the sequence $\{g_i:=f^{-1}_i: \Omega_i\to M \}$ converges uniformly on compact subsets of $\Omega_0$ to a holomorphic map $G:\Omega_0\to M$.  Then, one of the following assertions holds: 
        \begin{enumerate}
                \item[(i)] The sequence $\{f_i\}$ is compactly divergent, i.e., for each compact set $K\subset A_0$ and each compact set $L\subset \Omega_0$, there exists an integer $i_0$ such that $f_i(K)\cap L=\emptyset$ for $i\geq i_0$; or
                \item[(ii)] There exists a subsequence $\{f_{i_j}\}\subset \{f_i\}$  such that the sequence $\{f_{i_j}\}$ converges uniformly on compact subsets of $A_0$ to a biholomorphic map $F: A_0 \to \Omega_0$.
        \end{enumerate}
\end{proposition}
\begin{remark}
	In \cite{DN09}, Do Duc Thai and the first author proved the above proposition for the case that $M=N$, but the same proof can give the result in the above proposition. 
\end{remark}

Now let us recall the definition of the Kobayashi infinitesimal pseudometric. For a point $p, q$ be two points in a complex manifold $M$ and let $X\in T_pM$, the \emph{Kobayashi infinitesimal pseudometric} $F_M (p, X)$ is defined by 
$$
F_M(p,X)=\inf \{\alpha>0~|~\exists g\in \mathrm{Hol}(\Delta,\Omega), g(0)=p, g'(0)=X/\alpha\}.
$$
The \emph{ Kobayashi pseudodistance} $d^K_M (p, q)$ is defined by 
 $$
 d^K_M  (p, q) = \inf \int_a^b F_M(\gamma (t), \gamma'(t))dt, 
 $$
where the infimum is taken over all differentiable curves $\gamma : [a,b]\to M$ such that $\gamma(a)=p$ and $\gamma(b)=q$. A complex manifold $M$ is called Kobayashi hyperbolic if $d^K_M (p, q)\ne 0$ whenever $p\ne q$.

The following proposition is inspired by Theorem $3.2$ in \cite{FSi81}.
\begin{proposition}\label{key'-prop}
	Let $X, Y$ be two complex manifolds of dimension $n.$
	Let $\{ M_j \}_{j=1}^{\infty}$ be a sequence of domains in $X$ that converges to a domain $M$
	and $\{\Omega_j\}$ be a sequence of domains in $Y$ that converges to a taut domain $\Omega$.
	Let $\{f_j: M_j \to \Omega_j\}$ be a normal sequence of biholomorphisms.
	Suppose that there exists a point $a \in M$ satisfying the following conditions:
\begin{itemize}
	\item[(i)] $\lim\limits_{j \to \infty} f_j(a)=b \in \Omega$; 	
	\item[(ii)] $\varlimsup\limits_{j \to \infty} F_{M_j} (a, \xi)>0, \ \forall \xi \ne 0$.
\end{itemize}
	Then, $\{f_j\}$ contains a subsequence that converges uniformly on compacta to a biholomorphic map $f: M \to \Omega$.
\end{proposition}
\begin{remark}
	The condition $\mathrm{(ii)}$ always holds for the case that $X$ is contained in a Kobayashi hyperbolic domain in $\mathbb C^n$.
\end{remark}

We claim no originality for the following fact which is a slight extension of Hurwitz's theorem.
\begin{lemma} \label{hurwitz}
	Let $X,Y$ be complex manifolds of complex dimension $n$ and
	let $\{f_j\}$ be a sequence of injective holomorphic maps from $X$ to $Y.$  Suppose that $\{f_j\}$ converges locally uniformly to a
	non-constant holomorphic map $f:X \to Y.$ Then, $f$ is also injective on $X$.
\end{lemma}
\begin{proof}
	Suppose that there exist two distinct points $z, w \in X$ such that $f(z)=f(w)=\lambda \in Y$.
	Then we choose disjoint neighbourhoods $U,V$ of $z$ and $w,$ respectively, such that each of them is biholomorphic to the unit ball in $\mathbb C^n.$ By Hurwitz's theorem there exists $n_0$ such that $ \lambda \in f_{n_0} (U) \cap f_{n_0} (V)$. This is a contradiction to injectivity of $f_{n_0}.$
\end{proof}

\begin{proof}[Proof of Proposition \ref{key'-prop}]
	Since  $\{ f_j \}$ is normal and $f_j(a) \to b \in \Omega$ as $j\to \infty$, without loss of generality we may assume that $f_j$ converges uniformly on compacta to a holomorphic map $f\colon M \to \overline{\Omega}$ such that  $f(a)=b$.
	By the invariance property of the infinitesimal Kobayshi metric, we obtain
	$$
	F_{M_j} (a, \xi)= F_{\Omega_j} (f_j (a), f_j' (a) \xi),\; \forall j\geq 1, \,\forall \xi\in T_a^{\mathbb C}(M).
	$$
	Since $\Omega$ is taut, by letting $j \to \infty$ while applying Lemma $3.1$ in \cite{NT21} we obtain for each $\xi \ne 0$ that
	$$
	0<\varlimsup\limits_{j \to \infty} F_{M_j} (a, \xi) = F_{\Omega} (b, f' (a) \xi).
	$$
	It implies that $f'$ is invertible at $a$, and hence at every point of $M$ by Hurwitz's theorem. Therefore, by the inverse function theorem, $f$ is an open map on $M$, and so $f(M) \subset \Omega$.
	By Lemma \ref{hurwitz} $f$ is indeed one to one entirely on $M$.
	
	Finally, because of the biholomorphism from $M$ to $f(M) \subset \Omega$ and the tautness of $\Omega$, it follows that the sequence $f_j^{-1}: \Omega_j \to M_j\subset M$ is normal. Moreover, since $f_j(a) \to b \in \Omega$, it yields the sequences $f_j$ and $f_j^{-1}$ are not compactly divergent. Therefore, by Proposition \ref{T:7}, after passing to a subsequence, $f_j$ converges uniformly on compacta to a biholomorphic map $f: M \to \Omega$.
	
\end{proof}

By Proposition \ref{key'-prop}, we obtain the following corollary, which is a generalization of \cite[Lemma $1.1$]{Fr83}.
\begin{corollary}\label{key-prop}
	Let $\{ M_j \}_{j=1}^{\infty}$ be a sequence of domains in an $n$-dimensional Kobayashi hyperbolic complex manifold $M$ such that $\lim M_j = M$ and $\{\Omega_j\}$ be a sequence of domains in $\mathbb C^n$ converging to  a taut domain $\Omega\subset \mathbb C^n$.  Let $\{f_j: M_j \to \Omega_j\}$ be a normal sequence of biholomorphisms. If there exists a point $a \in M$ such that $f_j(a)$ converges to a point $b \in \Omega$, then $\{f_j\}$ contains a subsequence that converges uniformly on compacta to a biholomorphic map $f: M \to \Omega$.
\end{corollary}
\begin{proof}
We first prove that $F_{M_j} (a,\xi) \to F_M (a,\xi)$ for all $a \in M$ and for all and $\xi\in T^{\mathbb C}_a(M)$. Indeed, by the decreasing property of the infinitesimal Kobayashi distance we see that the sequence 
$\{F_{M_j} (a,\xi)\}_{j \ge 1}$ is decreasing and bounded from below by $F_M (a, \xi).$

On the other hand, for each $j \ge 1$, let $f_j: \Delta \to M$ be a holomorphic map with 
$$f_j(0)=a, f_j'(0)= \lambda \xi, \frac{1}{\lambda} \leq F_M (a,\xi)+\frac{1}{j}.$$
Notice that $f_j ((1-\frac{1}{j}) \Delta)$ is relatively compact in $M$ for any $j\geq 1$, and thus it is included in some $M_{k(j)}$. 
By considering the map $\tilde f_j: \Delta \to M_{k(j)}$ defined by
$$\tilde f_j (z):= f_j \Big ((1-\frac{1}{j}) z\Big ), \ z \in \Delta,
$$
we obtain 
$$F_{M_{k(j)}} (a, \xi) \leq \frac{1}{\lambda \Big (1-\frac{1}{j} \Big)}, \; j\geq 1.$$
Therefore, one has that
$$
\Big (1-\frac{1}{j} \Big) F_{M_{k(j)}} (a, \xi) \leq F_M (a,\xi)+\frac{1}{j},  \; j\geq 1.
$$
By letting $j \to \infty$, we get
$$
\lim\limits_{j \to \infty} F_{M_j} (a,\xi) \leq F_{M} (a,\xi).
$$
Hence, $\lim\limits_{j \to \infty} F_{M_j} (a,\xi)=F_{M} (a,\xi)$ for all $a\in M$ and $\xi\in T^{\mathbb C}_a(M)$, as desired.
The proof now follows from Proposition \ref{key'-prop}.
\end{proof}

\section{Proofs of Theorem \ref{maintheorem2} and Theorem \ref{maintheorem10}}\label{S4}
This section is devoted to the proofs of Theorem \ref{maintheorem2} and Theorem \ref{maintheorem10}. 
For this purpose, we first recall the following lemma which is essentially well-known~(cf.~\cite{NNTK19}).
\begin{lemma}[see \cite{NNTK19}]\label{dilation-property} Let $P$ be a weighted homogeneous polynomial with weight $(m_1,\ldots,m_{n-1})$ given by  \eqref{series expression} such that $P(z')> 0$ for all $z'\in \mathbb C^{n-1}\setminus\{0'\}$. Then, $\mathrm{Aut}(D_P)$ contains the following automorphisms $\phi_{a,\theta}$, defined by
\begin{equation}\label{formautomorphism}
(z',z_n)\mapsto  \left( \frac{(1-|a|^2)^{1/2m_1}}{(1-\bar{a}z_n)^{1/m_1}} z_1,\ldots,  \frac{(1-|a|^2)^{1/2m_{n-1}}}{(1-\bar{a}z_n)^{1/m_{n-1}}} z_{n-1}, e^{i\theta} \frac{z_n-a}{1-\bar{a}z_n}\right),
\end{equation}
where $a\in\Delta:=\{z\in \mathbb C\colon |z|<1\}$ and $\theta\in \mathbb R$. 
\end{lemma}

Next, we need the following lemma which is a generalization of  \cite[Lemma $2.5$]{Liu18}.
\begin{lemma}\label{exhaustion-ellipsoid} Let $\{\psi_j\}\subset \mathrm{Aut}(D_P)$ be a sequence of automorphisms 
	$$      
	\psi_j(z,w)=\left(\frac{(1-|a_j|^2)^{1/2m_1}}{(1-\bar{a}_j z_n)^{1/m_1}} z_1,\ldots,  \frac{(1-|a_j|^2)^{1/2m_{n-1}}}{(1-\bar{a}_jz_n)^{1/m_{n-1}}} z_{n-1}, \dfrac{z_n+a_j}{1+\bar a_j z_n}\right),
	$$
where $\{a_j\}\subset \Delta$ is a given sequence converging parabolically nontangentially to $1\in \partial \Delta$. Then, for any $s\in (0,1)$ we have $\psi_j^{-1}(D_P^s)\to D_P$ as $j\to \infty$. 
\end{lemma}
\begin{remark}\label{re-3.1} We have that $D_P^s\subset D_P$. Indeed, let $z\in D_P^s$ be arbitrary. Then, one sees that
$$
|z_n-1|^2+ 2s\mathrm{Re}(z_n-1)+sP(z')<0,
$$ 
or equivalently
$$
\frac{1}{s}\ |z_n-1|^2+ 2\mathrm{Re}(z_n-1)+P(z')<0.
$$ 
Since $0<s<1$, it follows that 
$$
 |z_n-1|^2+ 2s\mathrm{Re}(z_n-1)+sP(z')\leq \frac{1}{s}\  |z_n-1|^2+ 2\mathrm{Re}(z_n-1)+P(z')<0,
$$ 
which implies that $z\in D_P$.
\end{remark}

\begin{remark} In \cite[Lemma $2.5$]{Liu18}, B. Liu consider the case that $\mathrm{Im}(a_j)=0$ and $P(z')=|z'|^2$, i.e. $D_P$ is the unit ball $\mathbb B^n$ and $\mathcal B_s$ is the ball center at $(0', b)$ with radius $s=1-b$. However, the limit of $\psi_j^{-1}(\mathcal B_s)$ is exactly the ellipsoid $\Big\{|z_n|^2+ \dfrac{1}{1-b}|z'|^2<1\Big\}$, which is strictly smaller than the unit ball $\mathbb B^n$. Of course, according to Lemma \ref{exhaustion-ellipsoid} the limit of $\psi_j^{-1}(D_P^s)$ is  $\mathbb B^n$.	
\end{remark}

Now, in order to prove Lemma \ref{exhaustion-ellipsoid}, one needs the following lemma.
\begin{lemma}\label{parabol-convergence} Let  $\{a_j\}$ be a sequence in $\Delta $ such that $\displaystyle \lim_{j\to\infty} \frac{\left(\mathrm{Im}(a_j)\right)^2}{1-\mathrm{Re}(a_j)}=\alpha \in [0,2)$ and $\displaystyle \lim_{j\to\infty} a_j=1$. Then we have
\begin{itemize}
\item[(i)] $\displaystyle\lim_{j\to\infty} \frac{1-\mathrm{Re}(a_j)}{1-|a_j|^2}=\frac{1}{2-\alpha};$
\item[(iii)] $\displaystyle\lim_{j\to\infty} \frac{(1-\bar a_j)^2}{1-|a_j|^2}=\frac{-\alpha}{2-\alpha}.$
\item[(iii)] $\displaystyle\lim_{j\to\infty} \frac{|1-a_j|^2}{1-|a_j|^2}=\frac{\alpha}{2-\alpha}.$
\end{itemize}
\end{lemma}
\begin{proof} We have $x_j\to 0^+$, $y_j\to 0$, and $y_j^2/x_j\to \alpha$ as $j\to \infty$, where $x_j:=1-\mathrm{Re}(a_j), y_j:=\mathrm{Im}(a_j)$ as mentioned in the introduction. Moreover, a computation shows that
\begin{align*}
\frac{1-\mathrm{Re}(a_j)}{1-|a_j|^2}&=\frac{x_j}{1-(1-x_j)^2-y_j^2}=\frac{x_j}{2x_j-x_j^2-y_j^2}=\frac{1}{2-x_j-y_j^2/x_j};\\
\frac{(1-\bar a_j)^2}{1-|a_j|^2}&=\frac{(x_j+iy_j)^2}{1-(1-x_j)^2-y_j^2}=\frac{x_j^2-y_j^2+2i x_j y_j}{2x_j-x_j^2-y_j^2}=\frac{x_j-y_j^2/x_j+2iy_j}{2-x_j-y_j^2/x_j};\\
 \frac{|1-a_j|^2}{1-|a_j|^2}&=\frac{x_j^2+y_j^2}{1-(1-x_j)^2-y_j^2}=\frac{x_j^2+y_j^2}{2x_j-x_j^2-y_j^2}=\frac{x_j+y_j^2/x_j}{2-x_j-y_j^2/x_j},\;\forall j\geq 1.
\end{align*}
Therefore, the assertions follow since $x_j\to 0^+$ and $y_j^2/x_j\to \alpha$ as $j\to\infty$.
\end{proof}
\begin{remark} Let $\{a_j\}\subset \Delta$ be a sequence converging to $1\in \partial \Delta$, i.e. $x_j\to 1^+, y_j\to 0$ as $j\to \infty$. Moreover, $\displaystyle \frac{y_j^2}{x_j}< 2-x_j$ as $a_j \in \Delta$ for every $ j\geq 1$. Therefore, if the limit of $\displaystyle \lim_{j\to\infty} \frac{y_j^2}{x_j}$ exists, then its limit, say $\alpha$, is in $[0,2]$. One sees that our sequence $\{a_j\}$ converges ``very tangentially" to $1$ in $\Delta$ for the case $\alpha=2$.
\end{remark}

We note that Definition \ref{def1} deals with the case that $\alpha=0$. In this situation, by Lemma \ref{parabol-convergence} we obtain the following corollary.
\begin{corollary}\label{parabol-convergence1} Let  $\{a_j\}\subset \Delta$ be a sequence converging parabolically nontangentially to $1\in \partial \Delta$. Then we have
\begin{itemize}
\item[(i)] $\displaystyle\lim_{j\to\infty} \frac{1-\mathrm{Re}(a_j)}{1-|a_j|^2}=\frac{1}{2};$
\item[(ii)] $\displaystyle\lim_{j\to\infty} \frac{|1-a_j|^2}{1-|a_j|^2}=0.$
\end{itemize}
\end{corollary}
\begin{proof}[Proof of Lemma \ref{exhaustion-ellipsoid}] Let us recall that $b=1-s$ or $s=1-b\in (0,1)$.
	Then using the property (\ref{eq01}), a computation shows that
	\begin{align*}  
		&\hskip 0.5cm\left|\dfrac{z_n+a_j}{1+\bar a_jz_n}-b\right|^2+ s P\left(\frac{(1-|a_j|^2)^{1/2m_1}}{(1-\bar{a}_j z_n)^{1/m_1}} z_1,\ldots,  \frac{(1-|a_j|^2)^{1/2m_{n-1}}}{(1-\bar{a}_jz_n)^{1/m_{n-1}}} z_{n-1} \right) <s^2\\
		&\Leftrightarrow \left|\dfrac{z_n+a_j}{1+\bar a_jz_n}-b\right|^2+ s \dfrac{1-|a_j|^2}{|1+\bar a_jz_n|^2}P(z')    <s^2\\
		&\Leftrightarrow \left|\dfrac{z_n+a_j-b(1+\bar a_j z_n)}{1+\bar a_jz_n}\right|^2+ s\dfrac{1-|a_j|^2}{|1+\bar a_jz_n|^2}P(z')    <s^2\\
		&\Leftrightarrow \left|z_n+a_j-b(1+\bar a_j z_n)\right|^2+ s  (1-|a_j|^2)P(z')    <s^2|1+\bar a_jz_n|^2\\
		&\Leftrightarrow \left|z_n(1-\bar a_j b)+a_j-b\right|^2+ s  (1-|a_j|^2)P(z')    <s^2|1+\bar a_jz_n|^2\\
		&\Leftrightarrow |z_n|^2|1-\bar a_j b|^2+2\mathrm{Re}\left[(\bar a_j-b)(1-\bar a_j b)z_n\right]+|a_j-b|^2+ (1-b)  (1-|a_j|^2)P(z')  \\
		& \hskip0.5cm <s^2\left(|a_j|^2|z_n|^2+2\mathrm{Re}[\bar a_j z_n]+1\right)\\
		&\Leftrightarrow |z_n|^2\left(|1-\bar a_j b|^2-(1-b)^2|a_j|^2\right)+2\mathrm{Re}\left[\left((\bar a_j-b)(1-\bar a_j b)-(1-b)^2 \bar a_j\right)z_n\right]\\
		&\hskip0.5cm+ (1-b)  (1-|a_j|^2)P(z')   <(1-b)^2-|a_j-b|^2\\
		&\Leftrightarrow |z_n|^2+2\mathrm{Re}\left[\dfrac{(\bar a_j-b)(1-\bar a_j b)-(1-b)^2 \bar a_j}{|1-\bar a_j b|^2-(1-b)^2|a_j|^2}z_n\right]\\
		&\hskip0.5cm+ \dfrac{(1-b)  (1-|a_j|^2)}{|1-\bar a_j b|^2-(1-b)^2|a_j|^2}P(z')   <\dfrac{(1-b)^2-|a_j-b|^2}{|1-\bar a_j b|^2-(1-b)^2|a_j|^2}\\
		&\Leftrightarrow \left|z_n+\dfrac{(\bar a_j-b)(1-\bar a_j b)-(1-b)^2 \bar a_j}{|1-\bar a_j b|^2-(1-b)^2|a_j|^2}\right|^2+ \dfrac{(1-b)  (1-|a_j|^2)}{|1-\bar a_j b|^2-(1-b)^2|a_j|^2}P(z')  \\
		& \hskip0.5cm <\dfrac{(1-b)^2-|a_j-b|^2}{|1-\bar a_j b|^2-(1-b)^2|a_j|^2}+\left|\dfrac{(\bar a_j-b)(1-\bar a_j b)-(1-b)^2 \bar a_j}{|1-\bar a_j b|^2-(1-b)^2|a_j|^2}\right|^2.
	\end{align*}    
	Moreover, by a straightforward calculation, one has that

	\begin{align*}
		(\bar a_j-b)(1-\bar a_j b)-(1-b)^2 \bar a_j&=\bar a_j-b-\bar a_j^2 b+\bar a_j b^2-\bar a_j+2\bar a_j b-\bar a_jb^2=-b(1-\bar a_j)^2;\\
		(1-b)^2-|a_j-b|^2&=1-2b+b^2-|a_j|^2+2b\mathrm{Re}(a_j)-b^2\\
		&=1-|a_j|^2-2b(1-\mathrm{Re}(a_j));\\
		|1-\bar a_j b|^2-(1-b)^2|a_j|^2&=1-2\mathrm{Re}(a_j b)+|a_j|^2 b^2-|a_j|^2+2b |a_j|^2-b^2|a_j|^2\\
		&=1-|a_j|^2-2b\left( \mathrm{Re}( a_j)-|a_j|^2\right)\\
		&=1-|a_j|^2-2b\left( \mathrm{Re}( a_j)-1+1-|a_j|^2\right)\\
		&=(1-|a_j|^2)\left[1-2b\Big(1-\frac{1-\mathrm{Re}(a_j)}{1-|a_j|^2}\Big)\right].
	\end{align*}
	Hence, Corollary \ref{parabol-convergence1} implies that
	\begin{align*}
	\lim\limits_{j\to\infty} \dfrac{(\bar a_j-b)(1-\bar a_j b)-(1-b)^2 \bar a_j}{|1-\bar a_j b|^2-(1-b)^2|a_j|^2}&=\lim\limits_{j\to\infty}\dfrac{-b(1-\bar a_j)^2/(1-|a_j|^2)}{1-2b\left[1- (1-\mathrm{Re}( a_j))/(1-|a_j|^2)\right]} =0;\\
	\lim\limits_{j\to\infty} \dfrac{(1-b)  (1-|a_j|^2)}{|1-\bar a_j b|^2-(1-b)^2|a_j|^2}&=\lim\limits_{j\to\infty} \frac{(1-b)  (1-|a_j|^2)}{1-|a_j|^2-2b\left( \mathrm{Re}( a_j)-1+1-|a_j|^2\right)}\\
	&=\lim\limits_{j\to\infty}\dfrac{1-b}{1-2b\left[1- (1-\mathrm{Re}( a_j))/(1-|a_j|^2)\right]} =1;\\
	\lim\limits_{j\to\infty}\dfrac{(1-b)^2-|a_j-b|^2}{|1-\bar a_j b|^2-(1-b)^2|a_j|^2}&=\lim\limits_{j\to\infty} \frac{1-|a_j|^2-2b(1-\mathrm{Re}(a_j))}{1-|a_j|^2-2b\left( \mathrm{Re}( a_j)-1+1-|a_j|^2\right)}\\
	&=\lim\limits_{j\to\infty}\dfrac{1-2b(1-\mathrm{Re}(a_j))/(1-|a_j|^2)}{1-2b\left[1- (1-\mathrm{Re}( a_j))/(1-|a_j|^2)\right]} =1.\\
	\end{align*}
	This yields $\psi_j^{-1}(D_P^s)\to D_P$ as $j\to \infty$.
\end{proof}

We notice that the sequence $\psi_j^{-1}(D_P^s)$ would not converge to $D_P$ as $j\to \infty$ if $\{a_j\}\subset \Delta$ does not converge parabolically nontangentially to $1\in \partial \Delta$. More precisely, we have the following lemma.
\begin{lemma}\label{non-parabol} Let $\{a_j\}\subset \Delta$ be a sequence such that $\displaystyle\lim_{j\to\infty} a_j=1$ and $\displaystyle \lim_{j\to\infty} \frac{y_j^2}{x_j}=\alpha\in [0,2)$. Then, for any $s\in (0,1)$ we have that  $\psi_j^{-1}(D_P^s)$ converges to $D_P^s(\alpha)$, given by
\begin{align*}
\displaystyle D_P^s(\alpha)&=\Big\{z\in \mathbb C^n\colon \left|z_n+\frac{b\alpha}{(1-b)(2-\alpha)+b\alpha}\right|^2+ \frac{(1-b)(2-\alpha)}{(1-b)(2-\alpha)+b\alpha}P(z')  \\
		& \hskip2.5cm <\frac{2-\alpha -2b}{(1-b)(2-\alpha)+b\alpha}+\left|\frac{b\alpha}{(1-b)(2-\alpha)+b\alpha}\right|^2\Big\},	
\end{align*}
where $b=1-s$.
\end{lemma}
\begin{proof} By following the proof of Lemma \ref{exhaustion-ellipsoid} and by Lemma \ref{parabol-convergence}, we have

\begin{align*}
	\lim\limits_{j\to\infty} \dfrac{(\bar a_j-b)(1-\bar a_j b)-(1-b)^2 \bar a_j}{|1-\bar a_j b|^2-(1-b)^2|a_j|^2}& =\frac{b\alpha}{(1-b)(2-\alpha)+b\alpha};\\
	\lim\limits_{j\to\infty} \dfrac{(1-b)  (1-|a_j|^2)}{|1-\bar a_j b|^2-(1-b)^2|a_j|^2}& =\frac{(1-b)(2-\alpha)}{(1-b)(2-\alpha)+b\alpha};\\
	\lim\limits_{j\to\infty}\dfrac{(1-b)^2-|a_j-b|^2}{|1-\bar a_j b|^2-(1-b)^2|a_j|^2}&=\frac{2-\alpha -2b}{(1-b)(2-\alpha)+b\alpha}.
	\end{align*}
	Therefore, this yields $\psi_j^{-1}(D_P^s)\to D_P^s(\alpha)$ as $j\to \infty$, as desired.
\end{proof}

\begin{remark}\label{remark2023-1} Let $\{a_j\}\subset \Delta$ be as in Lemma \ref{non-parabol}. Then, $\psi_j^{-1}(D_P^s)$ converges to $D_P^s(\alpha)$, which strictly smaller than $D_P$ for $\alpha \in (0,2)$, and hence the limit of $\psi_j^{-1}(\Omega_j)$ may not be $D_P$.
\end{remark}

\begin{proof}[Proof of Theorem \ref{maintheorem2}] 
Let $\{\eta_j:=f_j^{-1}(a)\}\subset D^{s,r}_P$ be a sequence converging  to $(0',1)$ for some fixed $r\in (0,1)$. For simplicity, let us denote by $a_j=\eta_{jn}$ for $j\geq 1$. By assumption, we have $\{a_j\}\subset \Delta$ is a sequence converging parabolically nontangentially to $1\in \partial \Delta$.

We now consider the sequence of automorphisms $\{\psi_j\}\subset \mathrm{Aut}(D_P)$, given by
	$$  \displaystyle    
	\psi_j(z)=\left(\frac{(1-|a_j|^2)^{1/2m_1}}{(1-\bar{a}_j z_n)^{1/m_1}} z_1,\ldots,  \frac{(1-|a_j|^2)^{1/2m_{n-1}}}{(1-\bar{a}_jz_n)^{1/m_{n-1}}} z_{n-1}, \dfrac{z_n+a_j}{1+\bar a_j z_n} \right), \; j\geq 1.
$$
Then, Lemma \ref{exhaustion-ellipsoid} yields
	\begin{equation*}
		\begin{split}
			\lim\limits_{j\to\infty}\psi_j^{-1}(D_P^{s,r})= D_{P,r} ; \; \lim\limits_{j\to\infty}\psi_j^{-1}(\Omega_j)= D_P,
		\end{split}
	\end{equation*}
	where $D_{P,r}:=D_{P/r}=\Big\{z\in \mathbb C^{n}\colon |z_n|^2+ \dfrac{1}{r} P(z')<1\Big\}$. Moreover, since $\psi_j^{-1}(\eta_j)=\Big(\dfrac{\eta_{j1}}{\lambda_j^{1/2m_1}},\ldots, \dfrac{\eta_{j(n-1)}}{\lambda_j^{1/2m_{n-1}}} ,0\Big)\in D_{P,r}\cap \{z_n=0\}$, where $\lambda_j=1-|a_j|^2$ for all $j\geq 1$ and $D_{P,r}\cap \{z_n=0\}\Subset D_P\cap \{z_n=0\}$, by passing to a subsequence if necessary, we may assume that $\psi_j^{-1}(\eta_j)$ converges to some point $p\in D_P$ (see Figure \ref{fig1} below). Therefore, by Corollary \ref{key-prop} we conclude that $\psi_j^{-1}\circ f_j^{-1}$ converges uniformly on compacta to a biholomorphic map $F\colon M\to D_P$, and thus the proof is complete.     
\end{proof}

 In order to give a proof of Theorem \ref{maintheorem10}, we need the following theorem whose proof is a minor modification of the that of  \cite[Theorem $1.2$]{NT21}
 \begin{theorem}\label{togetmodelstronglypsc1} Let $M$ be an $n$-dimensional Kobayashi hyperbolic complex manifold and let $\Omega$ be a pseudoconvex domains in $\mathbb{C}^n$. Suppose that $\partial\Omega$ is $\mathcal{C}^2$-smooth boundary near a strongly pseudoconvex boundary point $\xi_0 \in \partial \Omega$. In addition, let $\{\Omega_j\}$ be a sequence of domains in $\mathbb C^n$ such that $\Omega_j \cap U = \Omega \cap U$, $j\geq 1$, for some neighborhood $U$ of $\xi_0$. Suppose also that $M$ can be exhausted by $\{ \Omega_j \}$ via an exhausting sequence $\{f_j: \Omega_j \to M_j\subset M\}$. If there exists a point $a \in M$ such that the sequence $\eta_j := f_j^{-1}(a)$ converges to $\xi_0$, then $M$ is biholomorphically equivalent to the unit ball $\mathbb{B}^n$.
  \end{theorem}
  Now we are ready to prove Theorem \ref{maintheorem10}.
\begin{proof}[Proof of Theorem \ref{maintheorem10}]
Suppose that $\eta_j := f_j^{-1}(a)$ converges $\Lambda$-tangentially to $(0',1)$ in $D_P$.  For simplicity, let us denote by $a_j=\eta_{jn}$. Then we consider the sequence of automorphisms $\{\psi_j\}\subset \mathrm{Aut}(D_P)$, given by
$$  \displaystyle    
	\psi_j(z)=\left(\frac{(1-|a_j|^2)^{1/2m_1}}{(1-\bar{a}_j z_n)^{1/m_1}} z_1,\ldots,  \frac{(1-|a_j|^2)^{1/2m_{n-1}}}{(1-\bar{a}_jz_n)^{1/m_{n-1}}} z_{n-1}, \dfrac{z_n+a_j}{1+\bar a_j z_n} \right), \; j\geq 1.
$$
 Since $\psi_j(0',0)=(0',a_j)\to (0',1)$ as $j\to\infty$ and the boundary point $(0',1)$ is of D'Angelo finite type, by \cite[Proposition 2.1]{Ber94} or \cite[Proposition 2.2]{DN09} it follows that
	\begin{equation*}
		\begin{split}
			\lim\limits_{j\to\infty}\psi_j^{-1}(\Omega_j) =\lim\limits_{j\to\infty}\psi_j^{-1}(\Omega_j\cap U)=\lim\limits_{j\to\infty}\psi_j^{-1}(D_P\cap U)= D_P.
		\end{split}
	\end{equation*}
	Let us set $b_j=\psi_j^{-1}(\eta_j)$ for all $j\geq 1$. Then, a straightforward computation shows that 
	$$
	b_j=\psi_j^{-1}(\eta_j)=\Big(\dfrac{\eta_{j1}}{\lambda_j^{1/2m_1}},\ldots, \dfrac{\eta_{j(n-1)}}{\lambda_j^{1/2m_{n-1}}} ,0\Big)\in D_{P}\cap \{z_n=0\}, 
	$$
	where $\lambda_j=1-|a_j|^2$ for all $j\geq 1$.
	
	Since $\{\eta_j\}$ converges $\Lambda$-tangentially to $(0',1)$ in $D_P$, it follows that there exists a sequence $\{r_j\}\subset (0,1)$ with $r_j\to 1$ as $j\to\infty$ such that
	$$
	|\eta_{jn}|^2+\dfrac{1}{r_j}P(\eta_j')\geq 1, \;\forall j\geq 1,
	$$
	or equivalently 
	$$
 P(\eta_j')\geq r_j(1-|a_j|^2), \;\forall j\geq 1.
	$$
	This implies that
	\begin{equation*}
		\begin{split}
				1> P(b_j')&=\dfrac{1}{\lambda_j} P(\eta_j')=\dfrac{1}{1-|a_j|^2} P(\eta_j')\geq r_j		\end{split}
	\end{equation*}
	for all $j\geq 1$. Therefore, we obtain that $P(b_j')\to 1$ as $j\to \infty$, and hence by passing to a subsequence if necessary, we may assume that $\psi_j^{-1}(\eta_j)$ converges to some strongly pseudoconvex  boundary point $p\in \partial D_P\cap\{z_n=0\}$ (see Figure \ref{fig2} below). Thus, by Theorem \ref{togetmodelstronglypsc1} we conclude that $\psi_j^{-1}\circ f_j^{-1}$ converges uniformly on compacta to a biholomorphic map $F\colon M\to \mathbb B^n$, and thus the proof is complete. 
\end{proof}

\begin{figure}[htp]
	\begin{minipage}{0.5\textwidth}
		
        \def\svgwidth{\columnwidth}
        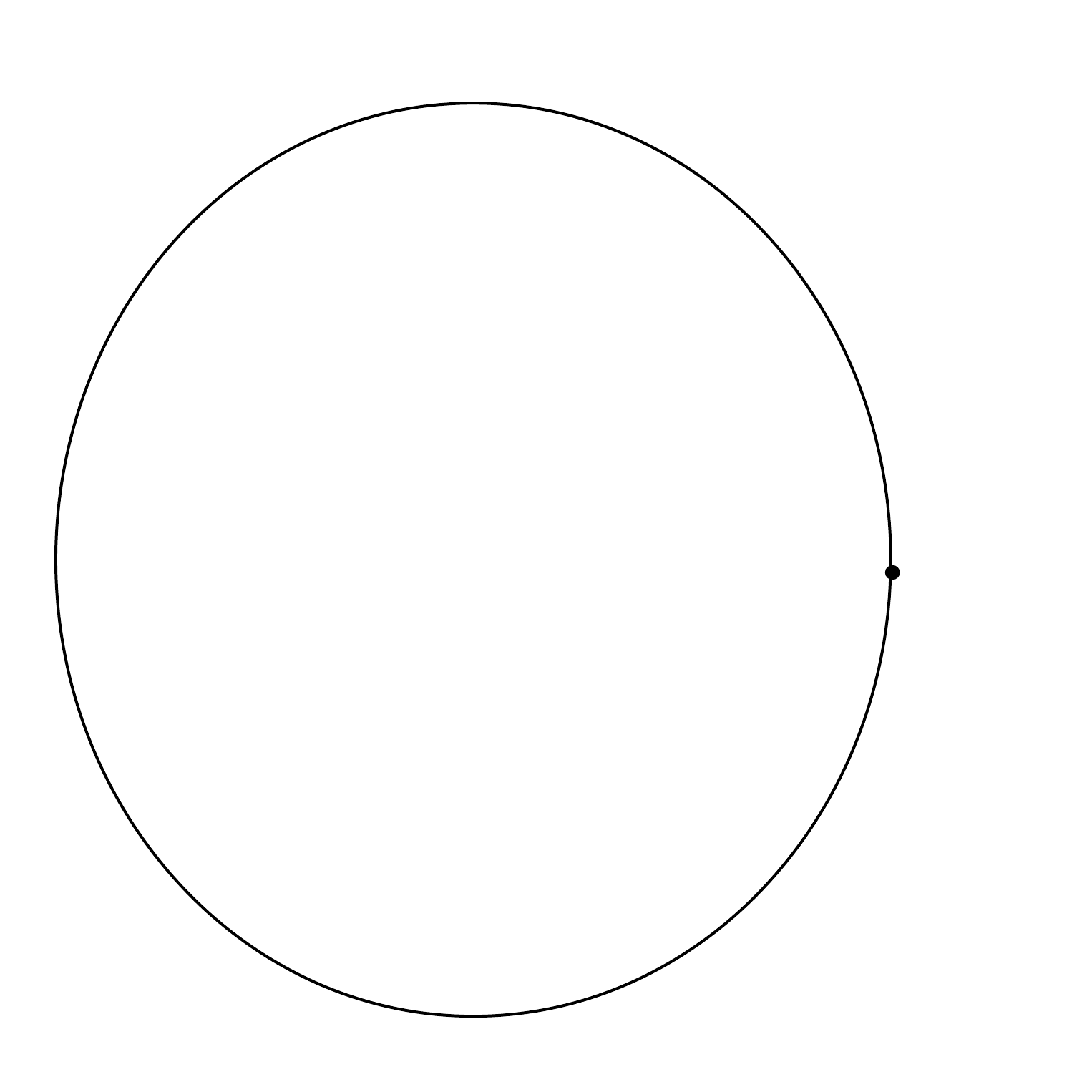

		\caption{$\Lambda$-nontangential convergence}\label{fig1}
	\end{minipage}%
	\begin{minipage}{0.5\textwidth}
		
        \def\svgwidth{\columnwidth}
        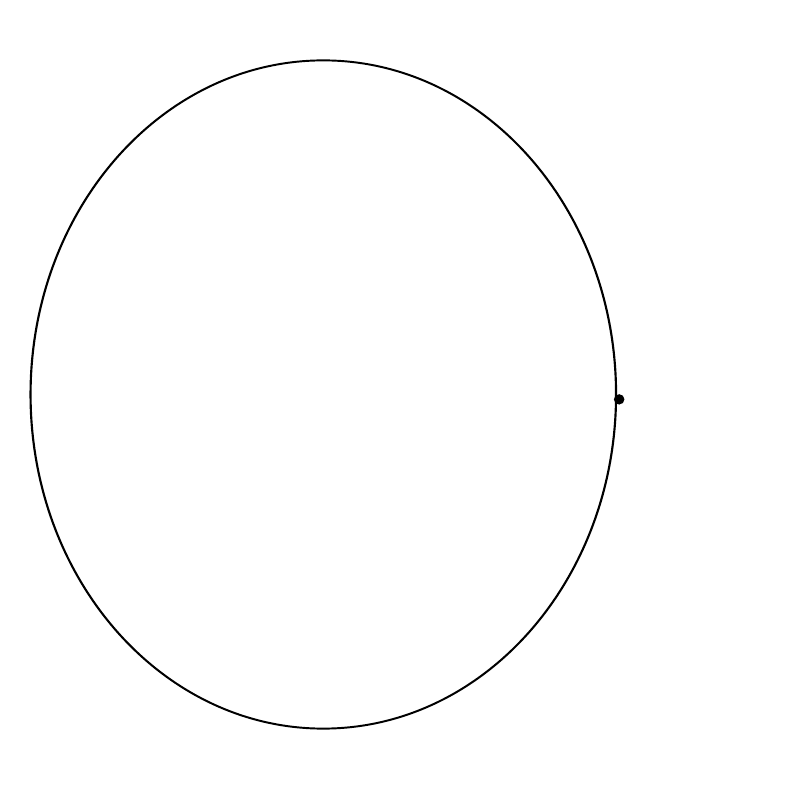

		\caption{$\Lambda$-tangential convergence}\label{fig2}
	\end{minipage}
\end{figure}

We now consider  a pseudoconvex domain $\Omega$ in $\mathbb C^n$ with noncompact automorphism group. Roughly speaking, there exists a sequence $\{\varphi_j\}\subset \mathrm{Aut}(\Omega)$ such that $\eta_j:=\varphi_j(a)$ converges to a boundary point $\xi_0\in \partial\Omega$ for some $a\in \Omega$. In \cite{NN19}, the first and last authors showed that if $\xi_0\in\partial \Omega$ is an $h$-extendible boundary point and  $\eta_j:=\varphi_j(a)$ converges $\Lambda$-nontangentially to $\xi_0$ (cf. \cite[Definition $3.4$]{NN19}), then $\Omega$ is biholomorphically equivalent to the model 
$$
M_{P}:=\left\{z\in \mathbb C^n \colon \mathrm{Re}(z_n)+P(z')<0\right \}.
$$
However, we notice that our domain $\Omega$ is exhausted by $\Omega$ via the sequence $\varphi_j\colon \Omega\to\Omega$ (by \cite[Proposition 2.1]{Ber94} or \cite[Proposition 2.2]{DN09}). Moreover,  in the case that $P(z')>0$ whenever $z'\in \mathbb C^{n-1}\setminus \{0\}$ and $D_P$ is a WB-domain, by Theorem \ref{maintheorem2}, Lemma \ref{non-parabol}, and Theorem \ref{maintheorem10} we have the following corollaries.

\begin{corollary}\label{corollary-model}
	Let $\Omega$ be a subdomain of $D_P$ and $\Omega\cap U=D_P\cap U$ for a fixed neighborhood $U$ of $(0',1)$ in $\mathbb C^n$. Suppose that there exists a sequence $\{\varphi_j\}\subset \mathrm{Aut}(\Omega)$ such that $\eta_j:=\varphi_j(a)$ converges to $(0',1)$ for some $a\in \Omega$. Then, one of the following assertions holds:
	\begin{itemize}
		\item[(i)] $\Omega$ is biholomorphically equivalent to $M_P$;
		\item[(ii)]  $\Omega$ and $D_P$ are biholomorphically equivalent to $\mathbb B^n$.
	\end{itemize}
\end{corollary}
\begin{proof} For simplicity, let us denote by $a_j=\eta_{jn}$ and by $x_j:=1-\mathrm{Re}(a_j)>0, y_j:=\mathrm{Im}(a_j)$ for $j\geq 1$. Then we also consider the sequence of automorphisms $\{\psi_j\}\subset \mathrm{Aut}(D_P)$, given by
$$  \displaystyle    
	\psi_j(z)=\left(\frac{(1-|a_j|^2)^{1/2m_1}}{(1-\bar{a}_j z_n)^{1/m_1}} z_1,\ldots,  \frac{(1-|a_j|^2)^{1/2m_{n-1}}}{(1-\bar{a}_jz_n)^{1/m_{n-1}}} z_{n-1}, \dfrac{z_n+a_j}{1+\bar a_j z_n} \right), \; j\geq 1.
$$
 Since $\psi_j(0',0)=(0',a_j)\to (0',1)$ as $j\to\infty$ and the boundary point $(0',1)$ is of D'Angelo finite type, again by \cite[Proposition 2.1]{Ber94} or \cite[Proposition 2.2]{DN09} it follows that
	\begin{equation*}
		\begin{split}
			\lim\limits_{j\to\infty}\psi_j^{-1}(\Omega) =\lim\limits_{j\to\infty}\psi_j^{-1}(\Omega\cap U)=\lim\limits_{j\to\infty}\psi_j^{-1}(D_P\cap U)= D_P.
		\end{split}
	\end{equation*}

Now we consider the limit of $\{\psi_j^{-1}(\eta_j)\}$ as $j\to\infty$ in terms of the boundary behavior of $\{a_j\}$ tending to $1\in \partial \Delta$. Indeed, 	we first assume that $\{\eta_j\}\subset D^s_P$ for some $0<s<1$. Then as in Remark \ref{re-3.1}, one sees that
$$
|a_j-1|^2<-2s\mathrm{Re}(a_j-1), \; \text{ for } j\geq 1,
$$
or equivalently $ x_j^2+y_j^2<2s x_j, \; \text{ for } j\geq 1$. Therefore, after taking some subsequence we may assume that there exists 
$$
\alpha:=\lim\limits_{j\to\infty}\frac{y_j^2}{x_j}\leq 2s<2.
$$

Thus we consider the following three cases:

\noindent
{\bf Case 1.} $\alpha=0$ and $\{\eta_j\}\subset D_P^{s,r}$ for some $r\in (0,1)$.  Then by Theorem \ref{maintheorem2}, one has $\Omega$ is biholomorphically equivalent to $D_P$. 

\noindent
{\bf Case 2.} $\alpha=0$ and $\{\eta_j\}$ converges $\Lambda$-tangentially to $(0',1)$.  Then by Theorem \ref{maintheorem10}, one has $\Omega$ is biholomorphically equivalent to $\mathbb B^n$. 

\noindent
{\bf Case 3.} $0<\alpha<2$. Then by Lemma \ref{non-parabol}, we have 
\begin{equation*}
		\begin{split}
			\lim\limits_{j\to\infty}\psi_j^{-1}(D_P^s)= D_{P}^s(\alpha), 
		\end{split}
	\end{equation*}
	where \begin{align*}
\displaystyle D_P^s(\alpha)&=\Big\{z\in \mathbb C^n\colon \left|z_n+\frac{b\alpha}{(1-b)(2-\alpha)+b\alpha}\right|^2+ \frac{(1-b)(2-\alpha)}{(1-b)(2-\alpha)+b\alpha}P(z')  \\
		& \hskip2.5cm <\frac{2-\alpha -2b}{(1-b)(2-\alpha)+b\alpha}+\left|\frac{b\alpha}{(1-b)(2-\alpha)+b\alpha}\right|^2\Big\},
\end{align*}
where $b=1-s$. 

Furthermore, we observe that $\psi_j^{-1}(\eta_j)=\Big(\dfrac{\eta_{j1}}{\lambda_j^{1/2m_1}},\ldots, \dfrac{\eta_{j(n-1)}}{\lambda_j^{1/2m_{n-1}}} ,0\Big)\in D_P^s(\alpha)\cap \{z_n=0\}$, where $\lambda_j=1-|a_j|^2$ for all $j\geq 1$, and 
$$
D_P^s(\alpha)\cap \{z_n=0\}=\{(z',0)\in \mathbb C^n\colon P(z')<\frac{2-\alpha-2b }{2-\alpha -2b+b\alpha}<1 \}\Subset D_P\cap \{z_n=0\}.
$$
Hence, by passing to a subsequence if necessary, we may assume that $\psi_j^{-1}(\eta_j)$ converges to some point $p\in D_P$. Therefore, by Corollary \ref{key-prop} we conclude that $\psi_j^{-1}\circ f_j^{-1}$ converges uniformly on compacta to a biholomorphic map $F\colon \Omega\to D_P$.

Finally, the remaining possibility is that  $\{\eta_j\}\not\subset D^s_P$ for any $s\in (0,1)$. By passing to a subsequence if necessary, we may assume that $\{\eta_j\}$ converges $\Lambda$-tangentially to $(0',1)$.  Then again by Theorem \ref{maintheorem10}, one concludes that $\Omega$ is biholomorphically equivalent to $\mathbb B^n$.

We note that in the case that $\Omega$ is biholomorphically equivalent to $\mathbb B^n$, we may assume that $\Omega=\mathbb B^n$. Then, since $\mathbb B^n\cap U=\Omega\cap U=D_P\cap U$, it follows that $(0',1)$ is a strongly pseudoconvex point and thus the Wong-Rosay theorem shows that $D_P$ is biholomorphically equivalent to $\mathbb B^n$.

Altogether, the proof is complete.
 \end{proof}

\section{Spherically extreme boundary points}\label{S5}
We start with the following result which is the main step of the proof of Theorem \ref{togetmodelstronglypsc}.
\begin{proposition}\label{spherical-strongly-psc}
	Let $\Omega \subset \mathbb R^N$ be a domain with $\mathcal C^2$ smooth boundary near some point $p \in \partial \Omega$. Suppose that $p$ is a locally spherically extreme point. Then, $\Omega$ is strongly convex at $p$. 
\end{proposition}
\begin{proof}
We write $p=(p',p_N).$
By the assumption there is a point $c(p) \in \Omega$ and $r>0$ such that $B(c(p), r) \subset \Omega$
and $\partial B(c(p),r) \cap \partial \Omega=\{p\}.$
After a translation we may assume $c(p)$ is the origin, 
then by a permutation of coordinates and the implicit function theorem we may assume that
near $p, \Omega$ admits a local defining function $\rho$ taking the form
\begin{equation} \label{eq1}
 \rho(x) :=x_N-\varphi (x'), \ x'=(x_1, \cdots, x_{N-1})
\end{equation} 
where $\varphi$ is $\mathcal C^2$ smooth defined on a small neighbourhood $U$ of $p'=(p_1, \cdots, p_{N-1}) \in  \mathbb R^{N-1}$. 
Furthermore, after applying a reflection if necessary we may achieve that $p_N \ge 0$
Then for $x=(x', x_N) \in \partial \Omega$ with $x' \in U$ we have
\begin{equation} \label{eq11}
\Vert x \Vert^2=\psi (x'):= \Vert x'\Vert^2+\varphi (x')^2.
\end{equation}
Thus, by the assumption that $p$ is a locally spherical extreme point of $\partial \Omega,$ 
we see that the function $\psi$
attains its local maximum at the point $p'$. Hence
\begin{equation} \label{eq0}
\nabla \psi|_{(p')}=0.
\end{equation}
Now we suppose that $\partial \Omega$ is {\it not} strictly convex at $p.$ Then we can find a 
tangent vector $\lambda =(\lambda_1, \cdots, \lambda_N) \in T_p (\partial \Omega), \lambda \ne 0$ such that
$$\text{Hess}\ (\rho)(p,\lambda):=
\sum_{1 \le j, k \le N} \frac{\partial^2 \rho}{\partial x_j \partial {x_k}} (p) \lambda_j {\lambda_k} \le 0.$$
By (\ref{eq1}) we obtain
$$\text{Hess}\ (\varphi)(p',\lambda'):= \sum_{1 \le j, k \le N-1} \frac{\partial^2 \varphi}{\partial x_j \partial {x_k}} (p') \lambda_j  {\lambda_k}\ge 0.$$
Now we wish to exploit further the assumption that $\psi$ attains its local maximum at the point $(p').$
Then for all $t \in \mathbb R$ small enough we have 
\begin{equation} \label{eq2}
\psi (p') \ge \psi (p'+t\lambda'), \ \lambda'=(\lambda_1, \cdots, \lambda_{N-1}).
\end{equation}
Using Taylor expansion theorem on the right hand side while taking into account (\ref{eq0}) we get
\begin{equation} \label{eq3}
\psi (p'+t\lambda')=\psi (p')+ \text{Hess}(\psi)(p',\lambda') t^2 +o(t^2),
\end{equation}
where 
$$\text{Hess}(\psi)(p',\lambda') = \sum_{1 \le j, k \le N-1}\frac{\partial^2 \psi}{\partial x_j \partial {x_k}} (p') \lambda_j {\lambda_k}.$$
Putting (\ref{eq2}) and (\ref{eq3}) together we see that 
$$\text{Hess}(\psi)(p',\lambda') \le 0.$$
By direct computation, using (\ref{eq11}) and the fact that $\varphi (p')=p_N,$ we arrive at 
\begin{equation} \label{eq4}
\sum_{j=1}^{N-1} \lambda_j ^2+2p_N \text{Hess}\ (\varphi)(p',\lambda')+2\Big ( \sum_{j=1}^{N-1}
\frac{\partial \varphi}{\partial x_j} (p')\lambda_j \Big )^2 \le 0.
\end{equation}
Next, we look at the right most term on the left hand side. Since 
$(\lambda_1, \cdots, \lambda_N) \in T_p (\partial \Omega)$ we have
$$\sum_{j=1}^N 
\frac{\partial \rho}{\partial x_j} (p) \lambda_j=0.$$
In view of (\ref{eq1}) the above equation takes the form 
$$\sum_{j=1}^{N-1} 
\frac{\partial \varphi}{\partial x_j} (p)\lambda_j =\lambda_n.$$
Plugging this equation into (\ref{eq4}) we obtain 
$$\sum_{j=1}^{N-1} \lambda_j^2+2p_n \text{Hess}\ (\varphi)(p',\lambda')+2 \lambda_N^2 \le 0.$$
Since $p_N \ge 0$ and since $\text{Hess}\ (\varphi)(p',\lambda') \le 0$
we infer that $\lambda_1=\cdots=\lambda_N=0,$ which is absurd. So we are done.
\end{proof}

\begin{proof}[Proof of Theorem \ref{togetmodelstronglypsc}]
By applying Theorem \ref{togetmodelstronglypsc1} and Proposition \ref{spherical-strongly-psc}, we obtain the desired conclusion.
\end{proof}
\begin{remark}
Let $\Omega$ be a domain in $\mathbb C^n$ contained in a ball $U$ of finite radius such that $\partial \Omega$ touches $\partial U$ and $\partial \Omega$ is $\mathcal{C}^2$-smooth in a neighborhood of $p$, i.e., $p$ is a spherical extreme boundary point. Then, B. L. Fridman and D. Ma \cite[Theorem $1.3$]{FM95} pointed that $\Omega$ can exhaust the unit ball $\mathbb B^n$. To prove this, let 
$$
Q_k=\{z\in \mathbb C^n\colon 2\;\mathrm{Re}(z_n)+ |z_n|^2+k |z'|^2<0\}, k\in \mathbb N.
$$
Then $Q_1$ is the unit ball centered at $(0', -1)$, $Q_k$ is biholomorphically equivalent to the unit ball $\mathbb B^n$ for any $k\in \mathbb N$. Moreover, by following the proof of \cite[Theorem $1.3$]{FM95}, there exists $N\geq 4$ and there exists a biholomorphic image $\Omega_1\subset Q_1$ of $\Omega$ that has a defining function of the form
$$
\rho_1(z)=2\; \mathrm{Re}(z_n)+ |z_n|^2+N |z'|^2+o(|z|^2).
$$ 

We now consider the family $f_\epsilon\in \mathrm{Aut}(Q_1),  \; \epsilon>0$, given by
\[
\begin{cases}
	w_n&=\dfrac{\epsilon z_n}{2-\epsilon+(1-\epsilon)z_n}\\
	w'&=\dfrac{\sqrt{\epsilon(2-\epsilon)}}{2-\epsilon+(1-\epsilon)z_n}z'.
\end{cases}
\]
Then, one sees that $f_\epsilon(0',-1)=(0',-\epsilon)$ converges nontangentially to $\xi_0:=(0',0)$ as $\epsilon\to 0^+$ and $f_\epsilon^{-1}(\Omega_1)\to Q_N$ as $\epsilon\to 0^+$. Therefore, $\Omega_1$ exhausts $Q_N$, and hence $\Omega$ exhausts the unit ball $\mathbb B^n$. We would like to emphasize here that Theorem \ref{togetmodelstronglypsc} ensures that any Kobayashi hyperbolic complex manifold $M$ exhausted by $\Omega$ is biholomorphically equivalent to $\mathbb B^n$, even the sequence $\{\eta_j\}$, given in the statement of Theorem \ref{togetmodelstronglypsc}, converges tangentially to $\xi_0$.
\end{remark}
Notice that if $p\in \partial \Omega$  is a spherically extreme boundary point, then $\lim\limits_{\Omega \ni z\to p}\sigma_{\Omega}(z)=1$ 
(see \cite[Theorem $3.1$]{KZ16}). Hence, the above corollary easily follows from the following corollaries.
\begin{corollary}\label{hhr} Let $M$ be an $n$-dimensional Kobayashi hyperbolic complex manifold and let $\{\Omega_j\}$ be a sequence of domains in $\mathbb{C}^{n}$. Suppose that $M$ can be exhausted by $\{\Omega_j \}$ via an exhausting sequence $\{f_j: \Omega_j \to M_j\subset M\}$. Suppose also that there exists a point $a \in M$ such that  
$$      
\lim\limits_{j\to \infty}\sigma_{\Omega_j}(\eta_j)=1,
$$ 
where $\eta_j := f_j^{-1}(a)$ for all $j\geq 1$. Then, $M$ is biholomorphically equivalent to the unit ball $\mathbb{B}^{n}$.
\end{corollary}
\begin{proof}
By the assumption on the sequence $\{\eta_j\}$, there exists a sequence of injective holomorphic maps 
$G_j: \Omega_j \to \mathbb{B}^{n}$ such that $G_j (\eta_j)=0$ and 
$G_j (\Omega_j)$ exhausts $\mathbb{B}^{n}.$ Thus the sequence 
$$
\tilde G_j:=G_j \circ f_j^{-1}: M_j \to G_j (\Omega_j)
$$
satisfies $\tilde G_j (a)=0$ for all $j\geq 1$. By Montel theorem, the sequence is also normal. Thus, we may apply Corollary \ref{key-prop} to complete the proof.
\end{proof}

By Corollary \ref{hhr}, one obtains the following corollary.
\begin{corollary}\label{hhr1} Let $M$ be an $n$-dimensional Kobayashi hyperbolic complex manifold and let $\Omega$ be a domain in $\mathbb{C}^{n}$. Suppose that $M$ can be exhausted by $\Omega$ via an exhausting sequence $\{f_j: \Omega \to M_j\subset M\}$. Assume that there exists a point $a \in M$ such that the sequence $\eta_j := f_j^{-1}(a)$ converges to $\xi_0 \in \partial \Omega$ and 
	$$      
	\lim\limits_{q\to \xi_0}\sigma_{\Omega}(q)=1.$$ 
	Then, $M$ is biholomorphically equivalent to the unit ball $\mathbb{B}^{n}$.
\end{corollary}

\begin{acknowledgement}Part of this work was done while the authors were visiting the Vietnam Institute for Advanced Study in Mathematics (VIASM). We would like to thank the VIASM for financial support and hospitality. The first author was supported by the Vietnam National Foundation for Science and Technology Development (NAFOSTED) under grant number 101.02-2021.42. The third author was supported by the Vietnam National Foundation for Science and Technology Development (NAFOSTED) under grant number 101.02-2019.304. We are grateful to an anonymous referee for his/her very careful reading and constructive comments that significantly improve our exposition.
\end{acknowledgement}
\subsection*{Data availability} Data sharing not applicable to this article as no data sets were generated or analyzed during the current study.

\bibliographystyle{plain}

\begin{thebibliography}{99}
\bibitem[AGK16]{AGK16} T. Ahn, H. Gaussier, and K.-T. Kim, \textit{Positivity and completeness of invariant metrics}, J. Geom. Anal.
{\bf 26} (2) (2016), 1173--1185.
\bibitem[BBMV21]{BBMV21} G. P. Balakumar, D. Borah, P. Mahajan, and K. Verma, \textit{Limits of an increasing sequence of complex manifolds}, Ann. Mat. Pura Appl. (4) 202 (2023), no. 3, 1381--1410. 
\bibitem[Ber94]{Ber94} F. Berteloot, \textit{Characterization of models in $\mathbb C^2$ by their automorphism groups}, Internat. J. Math., {\bf 5} (1994), 619--634.
\bibitem[Ber06]{Ber06} F. Berteloot, \textit{M\'{e}thodes de changement d'\'{e}chelles en analyse complexe}, Ann. Fac. Sci. Toulouse Math. (6) {\bf 15} (2006), 427--483.
\bibitem[CP01]{CP01} B. Coupet and S. Pinchuk, \textit{Holomorphic equivalence problem for weighted homogeneous rigid domains in $\mathbb C^{n+1}$}, Complex analysis in modern mathematics (Russian), 57--70, FAZIS, Moscow (2001).  
\bibitem[DN09]{DN09}  Do Duc Thai and Ninh Van Thu, \textit{Characterization of domains in $\mathbb  C^n$ by their noncompact automorphism groups}, Nagoya Math. J. {\bf 196} (2009), 135--160.
\bibitem[DT04]{TM}  Do Duc Thai and Tran Hue Minh, \textit{Generalizations of the theorems of Cartan and Greene- Krantz to complex manifolds}, Illinois J. of Math. {\bf 48} (2004), 1367--1384.
\bibitem[DGZ12]{DGF12} F. Deng, Q. Guan, L. Zhang,  \textit{Some properties of squeezing functions on bounded domains}, Pacific J. Math. {\bf 257}(2), 319--341 (2012). 
\bibitem[DGZ16]{DGF16} F. Deng, Q. Guan, L. Zhang,  \textit{Properties of squeezing functions and global transformations of bounded domains}, Trans. Amer. Math. Soc.  {\bf 368}(4), 2679--2696 (2016). 
\bibitem[FN21]{FN21} J. E. Forn{\ae}ss and N. Nikolov, \textit{Strong localization of invariant metrics}, Math. Ann. 383 (2022), no. 1-2, 353--360.
\bibitem[FS77]{FS77} J. E. Forn{\ae}ss and E. L. Stout, \textit{Polydiscs in complex manifolds}, Math. Ann. {\bf 227} (1977), no. 2, 145--153. 

\bibitem[FSi81]{FSi81}  J. E. Forn{\ae}ss and N. Sibony,  \textit{Increasing sequences of complex manifolds}, Math. Ann. 255 (1981), no. 3, 351--360.
\bibitem[Fr83]{Fr83} B. L. Fridman, \textit{Biholomorphic invariants of a hyperbolic manifold and some applications}, Trans. Amer. Math. Soc. {\bf 276} (1983), no. 2, 685--698.
\bibitem[FM95]{FM95} B. L. Fridman and D. Ma, \textit{On exhaustion of domains}, Indiana Univ. Math. J. {\bf 44} (1995), no. 2, 385--395.
\bibitem[GK87]{GK} R. E. Greene and S.G. Krantz, \textit{Biholomorphic self-maps of domains}, Lecture Notes in Math.,  {\bf 1276} (1987), 136--207.
\bibitem[KZ16]{KZ16} K.-T. Kim, L. Zhang, On the uniform squeezing property and the squeezing function, Pac. J. Math.  {\bf 282} (2) (2016), 341--358.
\bibitem[Liu18]{Liu18} B. Liu, \textit{Two applications of the Schwarz lemma}, Pacific J. Math. {\bf 296} (2018), no. 1, 141--153.
\bibitem[NN19]{NN19} Ninh Van Thu and Nguyen Quang Dieu, \textit{Some properties of $h$-extendible domains in $\mathbb C^{n+1}$}, J. Math. Anal. Appl. {\bf 485} (2020), no. 2, 123810, 14 pp.. 
\bibitem[NNTK19]{NNTK19} Ninh Van Thu, Nguyen Thi Lan Huong, Tran Quang Hung, and Hyeseon Kim, \textit{On the automorphism groups of finite multitype models in $\mathbb C^{n}$}, J. Geom. Anal. {\bf 29} (2019), no. 1, 428--450.
\bibitem[NT21]{NT21}  Ninh Van Thu and Trinh Huy Vu, \textit{A note on exhaustion of hyperbolic complex manifolds}, Proc. Amer. Math. Soc. 150 (2022), no. 5 , 2083--2093.
\bibitem[NNC21]{NNC21}  Ninh Van Thu, Nguyen Thi Kim Son and Chu Van Tiep, \textit{Boundary behavior of the squeezing function near a global extreme point},  Complex Var. Elliptic Equ. 68 (2023), no. 3, 351--360.
\bibitem[Yu95]{Yu95} J. Yu, \textit{Weighted boundary limits of the generalized Kobayashi-Royden metrics on weakly pseudoconvex domains}, Trans. Amer. Math. Soc.  {\bf 347}(2) (1995), 587--614.
\end{thebibliography}

\end{document}